\documentclass[a4paper]{amsart}

\usepackage{latexsym} 
\usepackage{mathtools} 
\usepackage{amsmath, amssymb} 
\usepackage{amsthm}
\usepackage{enumerate}
\usepackage[hyperfootnotes=false]{hyperref}

\theoremstyle{plain}
\newtheorem{thm}{Theorem}[section]
\newtheorem{lem}[thm]{Lemma}

\newtheorem{prop}[thm]{Proposition}
\newtheorem{cor}[thm]{Corollary}
\newtheorem{qes}[thm]{Question}

\theoremstyle{definition}
\newtheorem{rem}[thm]{Remark}
\newtheorem{dfn}[thm]{Definition}
\newtheorem{exa}[thm]{Example}

\mathtoolsset{showonlyrefs=true}
\numberwithin{equation}{section}

\newcommand{\m}[1]{\mathbb{#1}}
\newcommand{\bs}{\bigskip}

\newcommand{\LO}{\text{LO}}
\newcommand{\In}{\in\m{Z}_{\geq0}}

\title{Left orderings on inductive limits of amalgamated free products}
\author{Chihaya Jibiki}
\address{Department of Mathematics, School of Science, Institute of Science Tokyo, 2-12-1 Ookayama, Meguro-ku, Tokyo 152-8550 Japan}
\email{chihaya.j@gmail.com}

\date{}

\begin{document}
\begin{abstract}
We study left orderings of countably generated groups. In particular, we construct left orderings of inductive limits of amalgamated free products by using isolated left orderings of the groups appearing in the inductive system. Moreover, we show that they are no longer isolated orderings. 
\end{abstract}
\maketitle

\section{Introduction}\label{sec:Intro}

Let $G$ be a group. A {\it left ordering} of $G$ is a total ordering invariant under left multiplication: $a<_{G}b$ implies $ga<_{G}gb$ for all $a,b,g\in G$. The theory of left orderings is a wide and active topic of research (see Clay and Rolfsen \cite{MR3560661}, Deroin, Navas and Rivas \cite{deroin2014groups} and Ghys \cite{MR1876932}).

In this theory, there is an important object called the {\it space of left orderings} $\LO(G)$ of $G$, which is the set of all left orderings of $G$. It is well-known that $\LO(G)$ has a natural topology: an open sub-basis of neighborhoods of an ordering $<_{G}$ in the space $\LO(G)$ consists of the subsets $U_{g}\coloneqq\{<'_{G}\in\LO(G)|1<'_{G}g\}$, where $g$ runs over all positive elements of $<_{G}$. An {\it isolated ordering} of $G$ is defined as an isolated point of $\LO(G)$.  Sikora \cite{MR2069015} proved that $\LO(G)$ is a Hausdorff, totally disconnected and compact space. Consequently, for a countable group $G$, $\LO(G)$ is homeomorphic to a Cantor set if and only if $G$ does not have any isolated ordering.
For over a decade, the existence of isolated orderings for several groups has been investigated: for example, free products, amalgamated free products, braid groups, etc \cite{MR3200370,MR3476136,MR2859890,MR1214782}. 

We define the {\it positive cone} of a left ordering $<_G$ of $G$ as the set of all positive elements of $<_G$. It follows that there exists a one-to-one correspondence between the set of positive cones and that of left orderings. In contrast to the case of finitely generated groups, we cannot treat left orderings of countably generated groups until now. There is a well-known fact which is useful to determine whether a left ordering of a finitely many generated group is isolated if the positive cone of a left ordering of a finitely generated group is finitely generated, then it is an isolated ordering. 
This fact does not work in the case of countably generated groups. On the other hand, some researchers have been exploring the possibility of the existence of countably many generated groups which has isolated orderings \cite{MR4055461}.

We focus on Dehornoy's work \cite{MR3200370} and Ito's work \cite{MR3476136}, which are rewritten as Theorem \ref{thm:Dehornoy} and \ref{thm:Ito} in this paper, respectively. The reason why we review their works both of them constructed interesting isolated orderings of finitely generated groups. Dehornoy \cite{MR3200370} discovered a natural left ordering of the braid group $B_{n}$, which is now called the {\it Dehornoy ordering}. Generalizing the Dehornoy ordering, Ito \cite{MR2998793} introduced the {\it Dehornoy-like ordering} of groups, and constructed isolated Dehornoy-like orderings of torus knot groups. Dehornoy \cite{MR3200370} developed a useful method of describing left orderings, and applied it to construct isolated orderings of various groups including amalgamated free products of finitely many torus knot groups. Ito \cite{MR3476136} constructed isolated orderings of groups obtained by partially central amalgamation. 

\medskip

The aim of this paper is to investigate whether or not there are isolated orderings of countably generated groups. In particular, we generalize the isolated left orderings of Dehornoy's work \cite[Corollary 8.5]{MR3200370} and Ito's work \cite[Theorem 1.1]{MR3476136}. 
For each isolated ordering constructed in Dehornoy's theorem and Ito's theorem, we construct the pair $( G_{(\infty)}=\varinjlim G_{(m)}, P_{(\infty)}=\varinjlim P_{(m)})$ of a positive cone and a group, where $ P_{(\infty)}$ and $\{P_{(m)}\}_{m\In}$ are positive cones of $G_{(\infty)}$ and $\{G_{(m)}\}_{m\In}$, respectively. In addition, all $P_{(m)}$ are positive cones of isolated orderings. The following theorems state that $ P_{(\infty)}$ are no longer isolated orderings of $G_{(\infty)}$.

\begin{thm}\label{thm:main1}
 Let $m$ be a non-negative integer. Let $G_{(m)}$ be the group with finite presentation
    \[
            G_{(m)} = \langle \,g_{-m},\dots,g_{0},\dots ,g_{m}\mid  g_{-m}^{k_{-m+1}}=g_{-m+1}^{l_{-m+1}},\dots,g_{m-1}^{k_{m}}=g_{m}^{l_{m}}\,\rangle,
    \]
    and $P_{(m)}$ the positive cone of $G_{(m)}$ defined as
    \[
            P_{(m)}=\langle\{\,g_{-m}\}\cup\{g_{-m}^{-k_{-m+1}+1}\dots g_{i-1}^{-k_{i}+1}g_{i}\mid -m+1\leq i\leq m\,\}\rangle^{+},
    \]
where $k_n$ and $l_n$ are coprime integers greater than one for every $n\in\mathbb{Z}$. Then the inductive limit $P_{(\infty)}= \varinjlim P_{(m)}$ of the inductive system $\{P_{(m)}\}_{m\geq0}$ is the positive cone of (a left ordering of) the inductive limit $G_{(\infty)}=\varinjlim G_{(m)}$ of the inductive system $\{G_{(m)}\}_{m\geq0}$, and it is approximated by their conjugate orderings. In particular, $P_{(\infty)}$ is not the positive cone of an isolated ordering of $G_{(\infty)}$.
\end{thm}

\begin{thm}\label{thm:main2}
 Let $n$ be an integer. Let $G_n$ be a finitely generated group, $z_n$ a non-trivial central element of $G_n$, $\langle z_n\rangle$ the subgroup of $G_n$ generated by $z_n$, and $P_n$ the positive cone of an isolated left ordering $<_n$ of $G_n$. We assume that an isomorphism $\varphi_n:\langle z_n \rangle \to \langle z_{n+1} \rangle$ is given for every $n$. Suppose that $G_n$, $<_n$ and $P_n$ satisfy the following conditions:
\begin{enumerate}[{\rm (i)}]
\item $P_n$ is generated by elements $g_{1,n},\dots,g_{j_n,n}$ of $G_n$ as a semigroup;
\item the inequality $g_{i,n}<_n z_n$ holds for every integer $i\in\{1,\dots,j_n\}$;
\item $\langle z_n\rangle$ is not isomorphic to $G_n$ for infinitely many $n$.
\end{enumerate}
Then the inductive limit $G_{(\infty)}$ of the inductive system of amalgamated free products of the system $\{G_n,\langle z_n\rangle,\varphi_n\}_{-m\leq n\leq m}$ admits a positive cone induced by $P_n$'s, which is approximated by their conjugate ordering. In particular, $P_{(\infty)}$ is not the positive cone of an isolated ordering of $G_{(\infty)}$.
\end{thm}

These theorems claim that if a family $\{P_{(m)}\}_{m\In}$ of positive cones of a  family $\{G_{(m)}\}_{m\In}$ of finitely generated groups is extended to the positive cone $ P_{(\infty)}$ of a left ordering of a countably generated group $G_{(\infty)}$, then $ P_{(\infty)}$ may not be the positive cone of an isolated ordering even if every $P_{(m)}$ is associated with an isolated orderings of $G_{(m)}$. In other words, the theorems suggest that there are obstacles to construct isolated orderings of countably generated groups by extending  known isolated orderings.

\medskip
This paper is organized as follows. In Section \ref{sec:Back}, we provide some backgrounds on the theory of left orderings. In Section \ref{sec:3}, we first treat left orderings of countably generated groups. Next, we define the inductive limit of amalgamated free products and prove Lemma \ref{lem:main}, which plays a major role in this paper. In Section \ref{sec:4}, we treat Dehornoy's and Ito's isolated orderings and prove the main theorems.

\section*{Acknowledgements}
The author thanks Hisaaki Endo, who is the author's supervisor for reading carefully and pointing out some errors, and Tetsuya Ito for his feedback and interest in this work. The author was supported by JST SPRING, Japan Grant Number JPMJSP2106.

\section{Backgrounds}\label{sec:Back}

In this section, we recall some standard facts about left orderings. We begin with the definition of the positive cone of a left ordering. 

\begin{dfn}
Given a group $G$, we let $<_{G}$ be a left ordering of $G$. We call $P\coloneqq \{\,g\in G\mid 1<_{G}g\,\}$ the {\it positive cone} of $<_{G}$. We call any element belonging to $P$ the {\it positive element} and any non-trivial element not belonging to $P$ the {\it negative element}. 
A group which admits a left ordering is called a {\it left-orderable} group.
\end{dfn}

Positive cones are characterized by the next proposition.

\begin{prop}\label{prop:basic}
The positive cone $P$ of a left ordering of a group $G$ has the following properties:

\begin{enumerate}[{\rm (i)}]
\item $P\cdot P\subset P$,
\item $ G=P\sqcup P^{{-1}}\sqcup\{1\}$.
\end{enumerate}
\end{prop}

Conversely, a subset $P$ of the group $G$ which satisfies the conditions (i) and (ii) induces  the following left ordering $<_{G}$:
\[
g<_{G}h\stackrel{\text{def}}{\Longleftrightarrow} g^{-1}h\in P
\]
for all $g,h\in G$. Then $P$ coincides with the positive cone of $<_G$.

\bigskip

Therefore, we identify a left ordering with its positive cone. Since positive cones satisfy the axioms of semigroup, if a positive cone $P$ is generated by a finite subset $\{g_{1},\dots,g_{n}\}$ of $G$, then we write $P=\langle g_{1},\dots, g_{n}\rangle^{+}$.

\bs
We will now give the elementary notion of positive cones. Fix a group $G$, a left ordering $<_{G}$ of $G$ and its positive cone $P$.

A subgroup $H$ of $G$ is said to be a {\it convex subgroup} of the pair $(G,<_{G})$ or $(G,P)$ if $h_{1}<_{G}g<_{G}h_{2}$ implies $ g\in H$ for all $g\in G $ and $ h_{1},h_{2}\in H$. Convex subgroups form a chain. That is, for every  pair of convex subgroups $H_{1},H_{2}$ of the pair $(G,<_{G})$, either $ H_{1}\subset H_{2}$ or $H_{1}\supset H_{2}$ holds. Moreover, any intersection of convex subgroups is also a convex subgroup. For any convex subgroup $C$ of $(G,<_{G})$, there exist two useful properties: the first is that for some positive elements $c\in C$ and $g\in G$, the inequality $1<_{G}g<_{G}c$ implies $g\in C$; the second is that for any element $c\in C$ and any integer  $n$, $c^{n}\in C$ implies $c\in C$ \cite[Section 2.1]{MR3560661}.

A {\it Conradian ordering} $<_{G}$ is a left ordering of $G$ which has a more restrictive condition. Namely, if inequalities $1<_{G}g_{1}\in G$ and $1<_{G}g_{2}\in G$ hold, then there exists a natural number $n$ such that $g_{1}<_{G}g_{2}g_{1}^{n}$. The {\it Conradian soul} is the maximal (with respect to inclusions) convex subgroup of $(G,<_G)$ such that the restriction of $<_G$ is a Conradian ordering.

A left ordering is said to be {\it Archimedean} if for every pair of positive elements $g_{1},g_{2}\in G$, there exists $n\in\m{N}$ such that $h<_{G}g^{n}$ holds. The Archimedean property imposes a strong condition because if a group $G$ has an Archimedean left ordering, then $G$ is abelian \cite[Section 2.1]{MR3560661}. 

It is known that there is a relationship between the Conradian  and Archimedean property. Let $C\subset D$  be convex subgroups of $G$ with respect to a left ordering $<_{G}$. Then the pair $(C,D)$ is called a {\it convex jump} if there is no convex subgroup strictly between them. Moreover, the pair $(C,D)$ is called a {\it Conradian jump} if $(C,D)$ is a convex jump, $C$ is a normal subgroup of $D$, and the quotient ordering of $D/C$ obtained from $<_{G}$ is Archimedean. There exists a useful fact: if every convex jump $(C,D)$ of the left ordering $<_{G}$ is a Conradian jump, then $<_{G}$ is a Conradian ordering \cite[Section 9]{MR3560661}.

A left ordering $<_{G}$ is said to be {\it discrete} if $<_{G}$ has a minimal positive element. If a left ordering $<_{G}$ is not discrete, then it is said to be {\it dense}.

A left ordering $<_G$ is said to be {\it $g$-cofinal} for $g\in G$ if for any $g'\in G$, there exists $n\in \mathbb{Z}$ such that $g^{-n}<_G g'<_G g^n$ holds.

A left ordering $<_G$ is said to be {\it g-right ordering} for $g\in G$ if for any $g_1,g_2\in G$, $g_1<_G g_2$ implies $g_1g<_G g_2g$.

For any $g\in G$, the left ordering $<_{g}$ is obtained from a left ordering $<_{G}$ by conjugation: for any $g_{1},g_{2}\in G$, 
\begin{align}
g_{1}<_{g}g_{2}&\stackrel{\text{def}}{\Longleftrightarrow}g_{1}g<_{G}g_{2}g\\
&\Longleftrightarrow g^{-1}g_{1}g<_{G}g^{-1}g_{2}g.
\end{align}
 The left ordering $<_{g}$ is called a {\it conjugate ordering} of $<_{G}$.


\bs

Proposition \ref{prop:basic} implies that the set of all left orderings of a group $G$, which is written as $\LO(G)$, can be regarded as a subset of the power set $2^{G}$. The power set $2^{G}$ has natural topology: the set $2$, which consists of two points, is equipped with discrete topology and the set $2^{G}$ is equipped with product topology. Therefore, LO$(G)$ is equipped with relative topology induced by the topology of $2^G$. A basis of neighborhoods of a left ordering $<_{G}\in\LO(G)$ is the family of the sets $U_{F}=\{<'_{G}\in\LO(G)\mid 1<'_{G}f\text{ for all}\,f\in F\}$, where $F$ runs over all finite subsets of $<_{G}$-positive elements of $G$. A left ordering $<_{G}$ is isolated in $\LO(G)$ if and only if there is a finite set $S\subset G$ such that $<_{G}$ is the only left ordering satisfying $1<_{G}s$ for all $s\in S$. Such a left ordering is called the {\it isolated ordering} of $G$. If $\LO(G)$ is an uncountable set, then the isolated orderings of $G$ are called {\it genuine}.

The space $\LO(G)$ is compact and totally disconnected. If the group $G$ is countable, then $\LO(G)$ is metrizable \cite{MR2069015}. As a consequence, the space $\LO(G)$ for a countable left-orderable group $G$ is homeomorphic to the Cantor set if and only if $G$ does not have any isolated ordering. For this reason, it is important to investigate whether or not a given group admits isolated orderings. Many researchers have dealt with this problem for various group until now: free products \cite{MR2859890}, braid groups \cite{MR1859702}, amalgamated free products \cite{MR2998793,MR3476136,MR3641836}, etc. 
But all these groups are finitely generated groups except free products.

We will exhibit two theorems of Dehornoy and Ito,  in which isolated orderings are constructed.

\begin{thm}[Dehornoy {\cite[Corollary 8.5]{MR3200370}}]\label{thm:Dehornoy}
For $d\geq2$ and $k_{2},l_{2},k_{3},l_{3},\dots,k_{d},l_{d}\geq1$, let $G$ be the group with finite presentation $\langle\, x_{1},x_{2},\dots ,x_{d}\mid x_{1}^{k_{2}+1}=x_{2}^{l_{2}+1},x_{2}^{k_{3}+1}=x_{3}^{l_{3}+1},\dots,x_{d-1}^{k_{d}+1}=x_{d}^{l_{d}+1}\,\rangle$.
 Then $G$ is left-orderable, and the subsemigroup of $G$ generated by
 \[
 x_{1},x_{1}^{-k_{2}}x_{2},\dots,x_{1}^{-k_{2}}\dots x_{i-1}^{-k_{i}}x_{i}
 \]
 is the positive cone of a left ordering of $G$ which is isolated in LO$(G)$.
\end{thm}

\begin{thm}[T.Ito {\cite[Theorem 1.1]{MR3476136}}]\label{thm:Ito}
Let $G$ and $H$ be finitely generated groups. Let $z_{G}$ be a non-trivial central element of $G$, and $z_{H}$ a non-trivial element of $H$.

Let $\mathcal{G}=\{g_{1},\dots g_{m}\}$ be a finite generating set of $G$ which defines an isolated left ordering $<_{G}$ of $G$. We take a numbering of elements of $\mathcal{G}$ so that $1<_{G}g_{1}<_{G}\dots<_{G}g_{m}$ holds. Similarly, let $\mathcal{H}=\{h_{1},\dots,h_{n}\}$ be a finite generating set of $H$ which defines an isolated left ordering $<_{H}$ of $H$ such that the inequalities $1<_{H} h_{1}<_{H}\dots <_{H}h_{n}$ holds.

We assume the cofinality assumptions {\rm [CF(G)]},{\rm [CF(H)]}, and the invariance assumption {\rm [INV(H)]}:
\begin{align}
{\rm [CF(G)]}\ \ &g_{i}<_{G}z_{G}\ \text{holds for all}\ i;\\
{\rm [CF(H)]}\ \ &h_{i}<_{H}z_{H}\ \text{holds for all}\ i;\\
{\rm [INV(H)]}\ \ &<_{H}\text{is a $z_{H}$-right ordering.}
\end{align}

Let $X=G*_{\m{Z}}H=G*_{\langle z_{G}=z_{H}\rangle}H$ be the amalgamated free product of $G$ and $H$ over $\m{Z}$. For $i=1,\dots,m$, let $x_{i}=g_{i}z_{H}^{-1}h_{1}$. Then we have the following results.
\begin{enumerate}[$(1)$]
\item The generating set $\{x_{1},\dots,x_{m},h_{1},\dots,h_{n}\}$ of $X$ defines an isolated left ordering $<_{X}$ of $X$.
\item The isolated ordering $<_{X}$ does not depend on a choice of generating sets $\mathcal{G}$ and $\mathcal{H}$. Thus, $<_{X}$ only depends on the isolated orderings $<_{G},<_{H}$ and the elements $z_{G},z_{H}$.
\item The natural inclusions $\iota_{G}:G\to X$ and $\iota_{H}:H\to X$ are order-preserving homomorphisms.
\item $1<_{X}x_{1}<_{X}\dots<_{X}x_{m}<_{X}h_{1}<_{X}\dots h_{n}<_{X}z_{H}=z_{G}$. Moreover, $z=z_{G}=z_{H}$ is $<_{X}$-cofinal and the isolated ordering $<_{X}$ is a $z$-right ordering.
\item Let $Y$ be a non-trivial proper subgroup of $X$. If $Y$ is $<_{X}$-convex, then $Y=\langle x_{1}\rangle$, the infinite cyclic group generated $x_{1}$.
\end{enumerate}
\end{thm}


\begin{rem}
In general, the above two theorems provide different isolated orderings for the following reason: the group $X$ of Theorem \ref{thm:Ito} necessarily has a cofinal element, while the group $G$ of Theorem \ref{thm:Dehornoy} may have no cofinal elements.
\end{rem}

\section{Countably generated groups}\label{sec:3}

In this section, we study left orderings of countably generated group. We consider general countably generated groups in Section \ref{subsec:1}. In Section \ref{subsec:2}, we construct the inductive limit $G_{(\infty)}$ of an inductive system $\{G_{(m)}\}_{m\In}$ of amalgamated free products, which appear in Thoerem \ref{thm:main1} and in Theorem \ref{thm:main2}.

\subsection{General countably generated groups}\label{subsec:1}

Let $G_{(\infty)}$ be a countably generated left-orderable group and $\{g_{i}\}_{i\in\m{N}}$ a countable generating set of $G_{(\infty)}$. Let $G_{(m)}$ be the subset of $G_{(\infty)}$ generated by $g_1,\dots,g_m$ for every $m\in\m{N}$. Since $G_{(i)}$ is a subgroup of $G_{(j)}$ for all $i\leq j$, we can consider $G_{(\infty)}$ as the inductive limit of  the inductive system $\{G_{(m)}\}_{m\in \m{N}}$.

\begin{prop}\label{prop:general}
Let $G_{(\infty)}$ be the group generated by a countable set $\{g_{i}\}_{i\in\m{N}}$ and $ P_{(\infty)}$ a subset of $G_{(\infty)}$. Let $G_{(m)}$ be the subset of $G_{(\infty)}$ generated by $g_1,\dots,g_m$  and $P_{(m)}\coloneqq  P_{(\infty)}\cap G_{(m)}$ for all $m\in\m{N}$.

\begin{enumerate}[{\rm (1)}]
\item The set $P_{(m)}$ is a positive cone of $G_{(m)}$ for all $m\in\m{N}$ if and only if the set $ P_{(\infty)}$ is a positive cone of $G_{(\infty)}$.
\item    Assume that either of (hence both of) the conditions in (i) holds. Then the positive cone $P_{(m)}$ of the group $G_{(m)}$ is Conradian for all $m\in\m{N}$ if and only if the positive cone $ P_{(\infty)}$of the group $G_{(\infty)}$ is Conradian.
\end{enumerate}

\end{prop}

\begin{proof}
(1) We will show that $P_{(m)}$ and $P_{(\infty)}$ satisfy the conditions (i) and (ii) of Proposition \ref{prop:basic}.

Suppose that $P_{(m)}$ is a positive cone of $G_{(m)}$ for all $m\in\mathbb{N}$. Given any $p_{1},p_{2}\in P_{(\infty)}$, we can find $m_{1},m_{2}\in\m{N}$ such that $p_{1}\in G_{(m_{1})},p_{2}\in G_{(m_{2})}$, respectively. We put $M=\max\{m_{1},m_{2}\}$. Since both two elements $p_{1},p_{2}$ belong to the positive cone $P_{(M)}$, we have $p_{1}p_{2}\in P_{(M)}\subset  P_{(\infty)}$. Thus $P_{(m)}$ satisfies the condition (i). Given any non-trivial $g\in G_{(\infty)}$, we can find $m\in\m{N}$ such that $g\in G_{(m)}$. Since $P_{(m)}$ is the positive cone of a left ordering of $G_{(m)}$, we have either $g\in P_{(m)}\subset P_{(\infty)}$ or $g\in (P_{(m)})^{-1}\subset P_{(\infty)}^{-1}$. Thus, $P_{(m)}$ satisfies the condition (ii).

Suppose that $P_{(\infty)}$ is a positive cone of $G_{(\infty)}$. Since $ P_{(\infty)}$ is a semigroup and $G_{(m)}$ is a group, $ P_{(m)}$ is a semigroup. Since $ P_{(\infty)}$ satisfies the condition (i), for all non-trivial $g\in G_{(m)}$, either $g\in  P_{(\infty)}$ or $ g\in P_{(\infty)}^{{-1}}$ follows. Namely, we have either $g\in  P_{(\infty)}\cap G_{(m)}=P_{(m)}$ or $ g\in P_{(\infty)}^{{-1}}\cap G_{(m)}=(P_{(m)})^{-1}$.
 
 (2) It is proved in the same way as the proof of (1). 
\end{proof}

Proposition \ref{prop:basic} shows the relation between the global positive cone $ P_{(\infty)}$ and the local positive cones $P_{(m)}$.

\subsection{Inductive limits of amalgamated free products}\label{subsec:2}
In this subsection, we show the key Lemma \ref{lem:main}, which will be used to extend Theorem \ref{thm:Dehornoy} and Theorem \ref{thm:Ito} to theorems for some countably generated groups. Observing the two theorems, we can construct countably generated groups $G_{(\infty)}$ as inductive limits of amalgamated free products. We show how to construct $G_{(\infty)}$.

\bigskip
Let $\{G_{n}\}_{n\in\m{Z}}$ be a family of groups, $\{A_{n}\}_{n\in\m{Z}}$ a family of subgroups of $\{G_{n}\}_{n\in\m{Z}}$ and $\{\varphi_{n}:A_{n}\to G_{n+1}\}_{n\in\m{Z}}$ a set of injective homomorphisms. Then we construct the inductive limit by the induction step.
\medskip

We first set $G_{(0)}\coloneqq G_{0}$.

If $G_{(m)}$is defined and $G_{m}$ and $G_{-m}$ are subgroups of $G_{(m)}$ for some $m\in\m{Z}_{\geq0}$, then the group $G'_{(m)}$ is defined as the amalgamated free product:
\[
G'_{(m)}\coloneqq G_{(m)}*G_{m+1}(A_{m}\xhookrightarrow{\varphi_{m}} G_{m+1}).
\]
Hence,  $G_{m+1}$ and $G_{(m)}$ are also subgroups of $ G'_{(m)}$. This implies that $\varphi_{-m-1}:A_{-m-1}\to G_{-m}$ induces $\varphi_{-m-1}:A_{-m-1}\to G'_{(m)}$.

Moreover, the group $G_{(m+1)}$ is defined as the amalgamated free product:
\[
G_{(m+1)}\coloneqq G_{-m-1}*G'_{(m)}(A_{-m-1}\xhookrightarrow{\varphi_{-m-1}} G'_{(m)}).
\]
Hence, $G_{-m-1}$ and $G'_{(m)}$ are also subgroups of $ G_{(m+1)}$, i.e. $G_{-m-1}$ and $G_{m+1}$ are subgroups of $G_{(m+1)}$. 
\medskip

Processing the above construction recursively, we obtain the inductive limit $G_{(\infty)}\coloneqq \varinjlim G_{(m)}$ called the {\it inductive limit of amalgamated free products} $\{G_{(m)}\}_{m\In}$ of the family $(G_{n},A_{n},\varphi_{n})_{n\in\m{Z}}$.
\bs

Let $\iota_{n}:A_{n}\hookrightarrow G_{n}$ be the inclusion for each $n\in\m{Z}$. The group $G_{(\infty)}$ is considered to be the amalgamated free product of countably many groups shown in Figure \ref{figure}.

\begin{figure}[h]
\begin{center}
\includegraphics[scale=0.6]{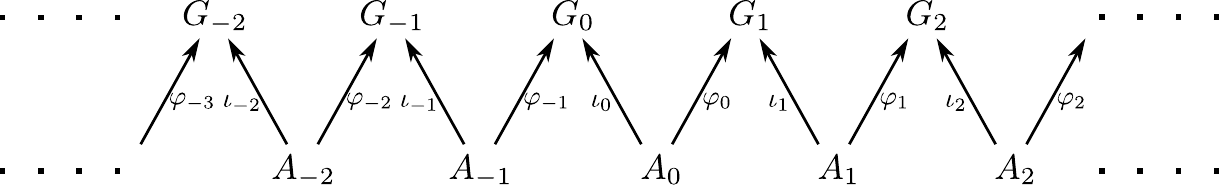}
\caption{ the inductive limit of amalgamated free products}\label{figure}
\end{center}
\end{figure}

We define the injective homomorphism $\varphi\coloneqq \bigcup_{n=-\infty}^{+\infty}\varphi_{n}:\bigcup_{n=-\infty}^{+\infty}A_n\hookrightarrow G_{(\infty)}$ by $\varphi(g_{n})=\varphi_{n}(g_{n})$ for $g_{n}\in A_{n}$. In addition, if a positive cone $P_{(m)}$ of $G_{(m)}$ and a positive cone $ P_{(\infty)}$ of $G_{(\infty)}$ are given, then we define $<_{m}$ as the left ordering induced by $P_{(m)}$ and $<_{\infty}$ as the left ordering induced by $ P_{(\infty)}$, respectively.

\bs

Our interest is in whether the inductive limit of amalgamated free products is left-orderable. We prove the next proposition.

\begin{prop}
Let $G_{(\infty)}$ be the inductive limit of amalgamated free products $\{G_{(m)}\}_{m\In}$ of the family $(G_{n},A_{n},\varphi_{n})_{n\in\m{Z}}$.
\begin{enumerate}[{\rm (i)}]
\item The group $G_{(\infty)}$ is left-orderable if and only if the group $G_{(m)}$ is left-orderable for all $m\In$.
\item The group $G_{(\infty)}$ is Conradian-orderable if and only if the group $G_{(m)}$ is Conradian-orderable for all $m\In$.
\end{enumerate}
\end{prop}

\begin{proof}
Let $\{g_{i,n}\}_{i\in I_{n}}$ be a set of generators of $G_{n}$ for all $n\in\m{Z}$. Then, it follows that $G_{(m)}$ is generated by $\{g_{i,j}\mid i\in I_{j},\,-m\leq j\leq m\}$ and $G_{(\infty)}$ is generated by $\{g_{i,j}\mid i\in I_{j},\, j\in\m{Z}\}$.  The proposition is proved in a similar way to the proof of Proposition \ref{prop:general}. 
\end{proof}

Bludov and Glass \cite[Theorem A]{MR2551464} gave a necessary and sufficient condition, called compatibility, for the inductive limit of amalgamated free products to be left-orderable in terms of the family $(G_{n},A_{n},\varphi_{n})_{n\in\m{Z}}$.

It is remarkable that inductive limits of amalgamated free products can be embedded in HNN-extensions.

\begin{prop}\label{prop:HNN}
Let $\{G_{n}\}_{n\in\m{Z}}$ be a family of copies of a group $G$, $\{A_{n}\}_{n\in\m{Z}}$ a family  of copies of a subgroup $A$ of $G$ and $\{\varphi_{n}:A_{n}\to G_{n+1}\}_{n\in\m{Z}}$ a set of injective homomorphisms. Let $ P_{(\infty)}$ be the positive cone of a left ordering of $G_{(\infty)}$.

Then, $G_{(\infty)}$ is embedded in the HNN-extension 
\[
H=\langle\, G,t\mid t^{-1}at=\varphi(a)\ \text{for all}\ a\in A\rangle.
\]
Moreover, $ P_{(\infty)}$ can be extended to two positive cones of left orderings of $H$.
\end{prop}
\begin{proof}[Sketch of proof] A similar assertion was proved by Bludov and Glass \cite[Theorem B]{MR2551464}.

Let $t$ be the automorphism of $G_{(\infty)}$ defined by $t(g_{n})=g_{n+1}$, where $g_{n}\in G_{n}$ and $g_{n+1}\in G_{n+1}$ represent  the same element of $G$. Then, $H$ is isomorphic to the semidirect product $G_{(\infty)}\rtimes \langle t\rangle$. Namely,
\[
1\to G_{(\infty)}\to H\to \langle t\rangle\to1
\]
is a split short exact sequence. Hence, $ P_{(\infty)}$ is extended to two left orderings of $H$ by the convex extension procedure \cite[Section 2.1.1]{deroin2014groups}.
\end{proof}

Next, we show three examples of inductive limits of amalgamated free products. 

\begin{exa}
\medskip
\noindent{\bf (a)}
For every $n\in\m{Z}$, we put $A_{n} =\{1\}$. Since the homomorphism $\varphi_{n}:A_{n}\hookrightarrow G_{n+1}$ is trivial, the inductive limit $G_{(\infty)}$ of amalgamated free products is isomorphic to the free product $*_{n\in\m{Z}}G_{n}$. Note that that all non-trivial free products do not have any isolated ordering \cite{MR2859890}.

\medskip
\noindent{\bf (b)}
Suppose that there exists a non-negative integer $m$ such that for all $n\geq m$ and $l\le -m$, $G_{n}=G_{l}=\{1\}$ holds. Then, the inductive limit $G_{(\infty)}$ of amalgamated free products is isomorphic to the amalgamated free product $*_{i\in I}G_{i}(A_{i}\xhookrightarrow{\varphi_{i}} G_{i+1})\,(I=\{1,\dots,m-2\})$ of finitely many groups. Hence $G_{(\infty)}$ is an ordinary amalgamated free product.

\medskip
\noindent{\bf (c)}For every $n\in\m{Z}$ and some $l\in\m{N}$, we set $G_{n}=A_{n}=\langle g_{n}\rangle\simeq\m{Z}$ and $\varphi_{n}(g_{n})=g_{n+1}^{l}$. Then, we have the inductive limit $G_{(\infty)}$ of amalgamated free products and its subgroups $G_{(m)}$, which have the presentation
\[
G_{(m)}=\langle\, g_{-m},\dots,g_{m}\mid g_{-m}=g_{-m+1}^{l},\dots,g_{m-1}=g_{m}^{l}\,\rangle.
\]
Therefore, the group $G_{(\infty)}$ is isomorphic to $ \m{Z}[l,l^{-1}]$. Namely, it is considered as a group ring. The group $G_{(\infty)}$ is a countably generated abel group and it has only two left orderings, which is non-genuine isolated orderings. The group $G_{(\infty)}$ can be embedded in an HNN-extension by Proposition \ref{prop:HNN}. In fact, $G_{(\infty)}$ is embedded in the Baumslag-Solitar group ${\rm BS}(1,l)$ and a split short exact sequence holds as below:
\begin{align}
&G_{(\infty)}\rtimes\langle t\rangle\simeq \text{BS}(1,l)=\langle\, g,t\mid t^{-1}gt=g^{l}\,\rangle,\\
&1\to G_{(\infty)}\to \text{BS}(1,l)\to\langle t\rangle\to 1.
\end{align}
Therefore, the left orderings of $G_{(\infty)}$ are extended to $\text{BS}(1,l)$.
\end{exa}

\bs

We introduce a notation of convex subgroups.

\begin{dfn}
Given the pair $(G,P)$ of a left-orderable group $G$ and a positive cone $P$ of a left ordering of $G$. For every subset $S$ of $G$, we define the group $\Gamma^{S}(P)$ called the {\it convex subgroup generated by $S$} as  the minimal convex subgroup of $(G,P)$ containing $S$.
\end{dfn}

\begin{rem}
The above definition of $\Gamma^{S}(P)$ is a generalization of the definition in \cite[Section 2.1.1]{deroin2014groups}, where the convex subgroup $\Gamma^{S}(P)$ was defined only if $S$ is a singleton.
\end{rem}

We show two propositions on convex subgroups generated by subsets of $G$.

\begin{prop}\label{prop:convex}
Let $G$ be a left-orderable group, $P$ a positive cone of $G$, and $S$ a subset of $G$.

\begin{enumerate}[{\rm (i)}]
\item  The convex subgroup $\Gamma^S(P)$ generated by $S$ is equal to the intersection of all convex subgroups of $(G,P)$ including $S$. 
\item For every convex subgroup $C$ including $S$ and every subgroup $T$ of $\Gamma^{S}(P)$, $T$ is a subgroup of $C$.
\end{enumerate}
\end{prop}
\begin{proof}
(i) Since  an arbitrary intersection of convex subgroups is convex, the intersection of all convex subgroups including $S$ is convex. The minimality is obvious.

(ii) $\Gamma ^{S}(P)$ is a minimal convex subgroup of $(G,P)$ which contains $S$. Thus, $\Gamma ^{S}(P)\subset C$ follows from $S \subset C$. Namely, $T\subset C$ holds.
\end{proof}

\bs

We show the next key lemma in order to prove main theorems.

\begin{lem}\label{lem:main}
Let $\{G_{n}\}_{n\in\m{Z}}$ be a family  of  countable left-orderable groups, $\{A_{n}\}_{n\in\m{Z}}$ a family such that $A_{n}$ is a subgroup of $G_{n}$ for every $n\in\m{Z}$ and infinitely many $A_{n}$ are not isomorphic to $G_{n}$.
Let $\{\varphi_{n}:A_{n}\to G_{n+1}\}_{n\in\m{Z}}$ be a set of injective homomorphisms and $G_{(\infty)}$ the inductive limit of amalgamated free products $\{G_{(m)}\}_{m\In}$ of the family $(G_{n},A_{n},\varphi_{n})_{n\in\m{Z}}$. 
Let $P_{(m)}$ be a positive cone of $G_{(m)}$ for every $m\In$, and $ P_{(\infty)}$ the inductive limit of the family $\{P_{(m)}\}_{m\In}$, which is a positive cone of $G_{(\infty)}$.

We assume that $ P_{(\infty)}$ is not Conradian and $P_{(m)}$ is discrete for every $m\In$. The minimal positive element of $P_{(m)}$ is denoted by $p_{m}$.  We also assume the following convexity assumptions:

\begin{enumerate}[{\rm(C1)}]
\item For every $m\In$, $G_{m+1}\subset \Gamma^{A_{m}}(P_{(m+1)})$;
\item For every $m\In$, $G_{m-1}\subset \Gamma^{\varphi(A_{m-1})}(P_{(m)})$;
\item For every $m\In$, there exists an integer $l\geq m$ such that $A_{l}\subset \Gamma^{p_{m}}(P_{(l)})$.
\end{enumerate}

Then the positive cone $ P_{(\infty)}$ is approximated by its conjugate orderings. In particular, $ P_{(\infty)}$ is not the positive cone of an isolated ordering of $G_{(\infty)}$.

\end{lem}

\begin{proof}
If the Conradian soul of $ P_{(\infty)}$ is trivial, then $ P_{(\infty)}$ is approximated by its conjugate orderings \cite[Theorem 3.7]{MR2551663}. Therefore, it is sufficient to prove that $ P_{(\infty)}$ has no non-trivial convex subgroups of $G_{(\infty)}$.

Let $C$ be a non-trivial convex subgroup of $(G_{(\infty)}, P_{(\infty)})$. There exists a non-trivial element $c\in C$ and $m\In$ such that $c$ belongs to $G_{(m)}$. Replacing $c$ with $c^{-1}$ if necessary, we can assume that $c$ is positive. The minimality of $p_m$ implies the inequality $1<_{m}p_{m}\leq_{m} c$ . Since $1\in C$ and $C$ is a convex subgroup, $p_{m}$ is contained in $C$. The condition (C3) implies that there exists an integer $l\geq m$ such that $A_{l}\subset \Gamma^{p_{m}}(P_{(l)})$.
Since $\Gamma^{p_{m}}(P_{(l)})$ is a subgroup of $C$,  $A_{l}$ is also a subgroup of $C$ by Proposition \ref{prop:convex}(ii).
Therefore, $\Gamma^{A_{l}}(P_{(l+1)})\subset C$ holds. The condition (C1) implies that $G_{l+1}\subset C$. 

Since $A_{l+1}\subset G_{l+1}$, the condition (C1) implies that $G_{l+2}\subset C$. Using 
the condition (C1) recursively, it is follows that $G_{i}\subset C$ for all $i\in\m{Z}_{\geq {l+1}}$.

Similarly, since $\varphi(A_{l})\subset G_{l+1}\subset C$, the condition (C2) implies that $G_{l}\subset C$. Using the condition (C2) recursively, it is follows that $G_{i}\subset C$ for all $i\in\m{Z}_{\leq l}$.

Consequently, the convex subgroup $C$ contains all groups $G_{i}$. Thus, $C\simeq G_{(\infty)}$ holds, which means that there exists no non-trivial convex subgroup of $G_{(\infty)}$. Since $P_{(\infty)}$ is not Conradian, the Conradian soul of $(G_{(\infty)}, P_{(\infty)})$ is a trivial group, that is, the positive cone $ P_{(\infty)}$ approximated by its conjugate orderings.
\end{proof}


\begin{rem}
In the ondition (C3), if $l$ is equal to $m$, then $A_{m}$ is a subgroup of $ \Gamma^{p_{m}}(P_{(m)})$. Since $p_{m}$ is the minimal positive element of $P_{(m)}$, $A_{m}\subset \Gamma^{p_{m}}(P_{(m)})=\langle p_{m}\rangle\simeq \m{Z}$ holds. 
\end{rem}


The next lemma is a variant of Lemma \ref{lem:main}.

\begin{lem}\label{lem:main2}
The statement of Lemma  \ref{lem:main} is also true if we replace the assumption that $P_{(\infty)}$ is not Conradian with either assumption {\rm (i)} or {\rm (ii)} exhibited below.
\begin{enumerate}[{\rm (i)}]
\item  $ P_{(\infty)}$ is not  Archimedean.
\item  $G_{(\infty)}$ is not abelian.
\end{enumerate}
\end{lem}

\begin{proof}

Although we assumed that $ P_{(\infty)}$ is not Conradian in Lemma \ref{lem:main}. it was used only to show the pair $(1,G_{(\infty)})$ is not a Conradian jump. Therefore, it is sufficient to prove that both conditions (i) and (ii) also imply that $(1,G_{(\infty)})$ is not a Conradian jump.

Supprose that $(1,G_{(\infty)})$ is a Conradian jump. Since the quotient ordering of $G_{(\infty)}/1=G_{(\infty)}$ is nothing but the left ordering induced by $P_{(\infty)}$, it is Archimedean. Hence the condition (i) implies that  $(1,G_{(\infty)})$ is not a Conradian jump.

If $P_{(\infty)}$ is Archimedean, then $G_{(\infty)}$ is abelian. It means that the condition (ii) implies the condition (i), and hence $(1,G_{(\infty)})$ is not a Conradian jump.

\end{proof}

\begin{rem}
If a positive cone of a left-orderable group $G$ is Archimedean, then there exists no non-trivial convex subgroups of $(G,P)$ \cite[Section 2.2]{MR3560661}. On the other hand, we consider the case that there exists a left ordering which is not Archimedean and has no non-trivial convex subgroups.
\end{rem}


\section{Proof of theorems}\label{sec:4}
In this section, we will apply propositions and lemmas in Section \ref{sec:3} to Dehornoy's and Ito's isolated orderings. 
More precisely, we first construct the inductive limit $G_{(\infty)}$ of an inductive system $\{G_{(m)}\}_{m\In}$ of amalgamated free products each of which is isomorphic to a group shown in Theorem \ref{thm:Dehornoy} or in Theorem \ref{thm:Ito}. We then construct the inductive limit $P_{(\infty)}$ of an inductive system $\{P_{(m)}\}_{m\In}$ of positive cones of $\{G_{(m)}\}_{m\In}$, and prove that $P_{(\infty)}$ is a positive cone of $G_{(\infty)}$. Although $P_{(m)}$ induces an isolated ordering of $G_{(m)}$ for each $m\In$, the left ordering of $G_{(\infty)}$ induced by $P_{(\infty)}$ is not isolated any longer.

\subsection{Inductive limits of Dehornoy's isolated orderings}\label{subsec:Dehornoy}
First of all, we construct the inductive limit of amalgamated free products each of which has a presentation described in Dehornoy's theorem (Theorem \ref{thm:Dehornoy}). 

\begin{lem}\label{lem:Dehornoy}
Let $n$ be an integer and $k_n$ and $l_n$ integers greater than one. Let
$G_n$ be the infinite cyclic group generated by $g_n$, $A_n$ the subgroup of $G_n$ generated by $g_n^{k_{n+1}}$, and $\varphi_n:A_n\to G_{n+1}$ the homomorphism defined by $\varphi_n(g_n^{k_{n+1}})= g_{n+1}^{l_{n+1}}$. Then the amalgamated free product $G_{(m)}$ constructed as in Subsection \ref{subsec:2} admits the presentation
\begin{align}
\langle\, g_{-m},\dots,g_{0},\dots ,g_{m}\mid g_{-m}^{k_{-m+1}}=g_{-m+1}^{l_{-m+1}},\dots,g_{m-1}^{k_{m}}=g_{m}^{l_{m}}\,\rangle,\label{pre:Dehornoy}
\end{align}
which is of the same type as the presentation stated in Theorem \ref{thm:Dehornoy}.
\end{lem}

\begin{proof}
The symbol $H_{(m)}$ denotes a group with presentation \eqref{pre:Dehornoy}. We prove by induction on $m$ that $G_{(m)}$ is isomorphic to $H_{(m)}$. If $m=0$, then we have $G_{(0)}=G_{0}=\langle g_{0}\rangle$ and $H_{(0)}=\langle g_0 \rangle$. Hence, we have $G_{(0)}=H_{(0)}$. Next, we assume that $G_{(i)}$ is isomorphic to $H_{(i)}$ for some $i\In$. By definition, $G'_{(i)}$ is isomorphic to $H_{(i)}*\langle g_{i+1}\rangle (g_i^{k_{i+1}}=g_{i+1}^{l_{i+1}})$, which admits the presentation
\begin{align}
\langle\, g_{-i},\dots,g_{i},g_{i+1}\mid g_{-i}^{k_{-i+1}}=g_{-i+1}^{l_{-i+1}},\dots,g_{i-1}^{k_{i}}=g_{i}^{l_{i}},g_{i}^{k_{i+1}}=g_{i+1}^{l_{i+1}}\,\rangle.
\end{align}
Then $G_{(i+1)}=\langle g_{-i+1}\rangle*G'_{(i)}(g_{-i-1}^{k_{-i}}=g_{-i}^{l_{-i}})$ is isomorphic to $H_{(i+1)}$. This completes the proof.
\end{proof}

\begin{rem}
By virtue of Theorem \ref{thm:Dehornoy} , the group $G_{(m)}$ admits the positive cone
\[
P_{(m)}=\langle\{g_{-m}\}\cup\{\,g_{-m}^{-k_{-m+1}+1}g_{-m+1}^{-k_{-m+2}+1}\dots g_{i-1}^{-k_{i}+1}g_{i}\mid -m+1\leq i\leq m\,\}\rangle^{+}.
\]
This positive cone induces an isolated ordering of $G_{(m)}$.  
\end{rem}

\medskip
We show some properties of the left ordering induced by $P_{(m)}$. 

\begin{prop}\label{prop:Dehornoy}
 For every non-negative integer $m$ and every integer $i$ with $-m\leq i \leq m$, we set $a_{i,m}\coloneqq g_{-m}^{-k_{-m+1}+1}g_{-m+1}^{-k_{-m+2}+1}\dots g_{i-1}^{-k_{i}+1}g_{i}$ and $a_{-m,m}\coloneqq g_{-m}$. The following equalities and inequalities hold.
\begin{enumerate}[{\rm (i)}]
\item $a_{i,m}=a_{i+1,m}g_{i+1}^{l_{i+1}-1}$,
\item $a_{i,m}=g_{-m-1}^{k_{-m}-1}a_{i,m+1}$,
\item $1<_{m}g_{i}<_{m}g_{i+1}$,
\item $a_{i+1,m}<_{m}a_{i,m}$,
\item $a_{i,m+1}<_{m+1}a_{i,m}$,
\item $1<_{m}a_{m,m}<_{m}a_{m-1,m}<_{m}\dots<_{m}a_{-m,m}=g_{-m}<_{m}\cdots<_{m} g_{m}$.
\end{enumerate}
\end{prop}

\begin{proof}
\begin{enumerate}[{\rm (i)}]
\item[(iii)] Fix an integer $m\geq0$.  For every $i\in\{-m,\dots,m\}$, since $a_{i,m+1}=a_{-m-1,m+1}^{-k_{-m}+1}a_{i,m}$ holds by (ii), we have $1<_{m+1}a_{i,m}$. Similarly, since $a_{i,m}=a_{-m,m}^{-k_{-m+1}+1}a_{i,m-1}$ holds, we have $1<_{m+1}a_{i,m-1}$. By induction, it follows that $1<_{m+1}a_{i,j}$ for all $j\in \{i,\dots,m\}$. It follows from $1<_ma_{-i+1,m}$ that $1<_{m}a_{-i+1,i}=g_{i}^{-k_{i+1}+1}g_{i+1}$. Therefore, we have $g_{i}^{k_{i+1}-1}<_{m}g_{i+1}$. This inequality show
\begin{align}
    1<_m a_{-m,m}=g_{-m}<_m g_{-m}^{k_{-m+1}-1}<g_{-m+1}.
\end{align}
By induction, we have $1<_{m}g_{i}<_{m}g_{i+1}$.
\item[(iv)] By (i) and (iii), we have $a_{i+1,m}^{-1}a_{i,m}=g_{i+1}^{l_{i+1}+1} >_{m}1$.
\item[(v)] By (i) and (iv), the following inequalities hold:
\begin{align}
a_{i,m+1}^{-1}a_{i,m}&=a_{i,m+1}^{-1}g_{-m-1}^{k_{-m}-1}a_{i,m+1}\\
&=(a_{i,m+1}^{-1}a_{-m-1,m+1})a_{-m-1,m+1}^{k_{-m}-2}a_{i,m+1}>_{m+1}1
\end{align}

\item[(vi)] These inequalities easily follow from (iii) and (iv).
\end{enumerate}
\end{proof}

\begin{lem}
The inductive limit $ P_{(\infty)}=\varinjlim P_{(m)}$ is well-defined and a positive cone of a left ordering of $G_{(\infty)}$.
\end{lem}

\begin{proof}
Since $a_{i,m}=a_{-m-1,m+1}^{k_{-m}-1}a_{i,m+1}>_{m+1}1$ holds by Proposition \ref{prop:Dehornoy} (ii), $P_{(l)}\subset P_{(m)}$ follows for all non-negative integers $l,m$ with $l\leq m$. Hence the inductive limit $ P_{(\infty)}$ of the inductive system $\{P_{(m)}\}_{m\In}$ is defined. It follows from Proposition \ref{prop:general}(1) that $ P_{(\infty)}$ is a positive cone of $G_{(\infty)}$.
\end{proof}

We have thus obtained the pair $( G_{(\infty)},P_{(\infty)})$ of inductive limits, where $ P_{(\infty)}$ and $\{P_{(m)}\}_{m\In}$ are positive cones of $G_{(\infty)}$ and $\{G_{(m)}\}_{m\In}$, respectively. In addition, all $P_{(m)}$ are positive cones of isolated orderings since they are finitely generated. We are ready to prove one of our main theorems.

\begin{proof}[Proof of Theorem \ref{thm:main1}]
We have already proved that $ P_{(\infty)}$ is a positive cone of $G_{(\infty)}$.
We apply Lemma \ref{lem:main} and Lemma \ref{lem:main2} (ii) to prove that $P_{(\infty)}$ induces a non-isolated ordering.

\smallskip
\paragraph{Claim 1.}$ G_{(\infty)}$ is not abelian.
\smallskip


Suppose that $G_{(\infty)}$ is abelian. For any non-negative integer $m$, $G_{(m)}$ is an amalgamated free product presented as in Theorem \ref{thm:main1} and abelian. Some $k_{i},l_{i}$ are equal to one for non-negative integer $i$ with $-m+1\leq i\leq m$. However, this contradicts the assumption. 

\smallskip
\paragraph{Claim 2.}$P_{(m)}$ is discrete for every $m\In$.
\smallskip

Suppose that there exists $g\in G_{(m)}$ such that $1<_{m}g<_{m}a_{m,m}$. We write $g=a_{i,m}p$, where $p$ is an element of $P_{(m)}$ and $i$ is a non-negative integer smaller than $m$. Then, we have $1<_{m}a_{i,m}p<_{m}a_{m,m}$. By multiplying by $a_{m,m}^{-1}$ on the left, we have $(a_{m,m}^{-1}a_{i,m})p<_{m}1$. From Proposition \ref{prop:Dehornoy} (iv), this is a contradiction. Therefore, the element $a_{m,m}$ is the minimal positive element of $P_{(m)}$, i.e. $P_{(m)}$ is discrete.



\smallskip
\paragraph{Claim 3.} The condition (C1) of Lemma \ref{lem:main} is satisfied.
\smallskip

Let $C$ be a convex subgroup of $(P_{(m+1)},G_{(m+1)})$ which contains $\Gamma^{\langle g_{m}^{k_{m+1}}\rangle}(P_{(m+1)})$. Then  $C$ contains $g_{m}^{k_{m+1}}$, which means $g_{m}\in C$ by convexity. Since $g_{m}^{k_{m+1}}=g_{m+1}^{l_{m+1}}$ holds, $C$ contains $g_{m+1}^{l_{m+1}}$. Thus, the group $C$ contains $g_{m+1}$, i.e. $\langle g_{m+1}\rangle\subset C$. By Proposition \ref{prop:convex} (i), $G_{m+1}\subset \Gamma^{A_{m}}(P_{(m+1)})$ holds.

\smallskip
\paragraph{Claim 4.}The condition (C2) of Lemma \ref{lem:main} is satisfied.
\smallskip

The proof is similar to that of Claim 3.

\smallskip
\paragraph{Claim 5.}The condition (C3) of Lemma \ref{lem:main} is satisfied.
\smallskip

Let $C$ be a convex subgroup of $(P_{(m+1)},G_{(m+1)})$ which contains $\Gamma^{a_{m,m}}(P_{(m+1)})$. Then  $C$ contains $a_{m,m+1}$ and $a_{m+1,m+1}$ by Proposition \ref{prop:Dehornoy}(iv) and (v). By Proposition \ref{prop:Dehornoy} (ii), since $g_{-m-1}^{k_{-m}-1}=a_{m,m}a_{m,m+1}^{-1}$ holds, $C$ contains $g_{-m-1}^{k_{-m}-1}$, i.e. $C$ contains $g_{-m-1}$. Since, $C$ contains $g_{-m-1}^{k_{-m}}$, $C$ contains $g_{-m}^{l_{-m}}$, which means that $C$ contains $g_{-m}$. We operate this recursive method. Consequently, $C$ contains $g_{m+1}$, which mean that $C$ contains $\langle g_{m+1}^{k_{m+2}}\rangle$ by Proposition \ref{prop:convex} (i).

\medskip
We complete to prove all conditions of Lemma \ref{lem:main}. Therefore, by Lemma \ref{lem:main}, this completes the proof.
\end{proof}

By Proposition \ref{prop:HNN}, we obtain a non-isolated ordering on an HNN-extension as below.

\begin{cor}\label{cor:Dehornoy}
We make the same assumption as in Theorem \ref{thm:Dehornoy}.

Then, $G_{(\infty)}$ is embedded in the HNN-extension $H$ with presentation
\begin{align}
\langle\, g,t\mid tat^{-1}=\varphi(a)\ \text{for all}\ a\in A_{m}, g\in G_{m}, m\In\,\rangle.
\end{align}
Moreover, $ P_{(\infty)}$ is extended to two positive cones of left orderings of $H$, which are approximated by its conjugate orderings, i.e. which induce non-isolated orderings.
\end{cor}
\begin{proof}
By Proposition \ref{prop:HNN},  the two positive cones of $H$ denoted by $\bar{P}_{1}$ and $\bar{P}_{2}$ are obtained from $ P_{(\infty)}$. Since $G_{(\infty)}$ has only one convex subgroup of $(\bar{P}_{1},H)$ and $(\bar{P}_{2},H)$ and $ G_{(\infty)}$ is not abelian, the Conradian souls of $(\bar{P}_{1},H)$ and $(\bar{P}_{2},H)$ are trivial. Namely, $\bar{P}_{1}$ and $\bar{P}_{2}$ are approximated by its conjugate orderings. In particular, these are not isolated orderings.
\end{proof}

\subsection{Inductive limits of Ito's isolated orderings}\label{subsec:Ito}

Next, we construct the inductive limit of amalgamated free products each of which is described in Ito's theorem (Theorem \ref{thm:Ito}). 

\begin{lem}\label{lem:Ito}
For every $n\in\m{Z}$, let $G_{n}$ be a finitely generated group, $z_{n}$ a non-trivial central element of $G_{n}$ and $P_{n}=\langle g_{1,n},\dots,g_{j_{n},n}\rangle^{+}$ the positive cone of an isolated left ordering $<_{n}$ of $G_{n}$. Let $A_n$ be the infinite cyclic group $\langle z_n\rangle$ generated by $z_n$, and $\varphi_n:A_n\to G_{n+1}$ the homomorphism defined by $\varphi_n(z_n)=z_{n+1}$. We assume that the $g_{i,n}<_{n}z_{{n}}$ holds for all $1\leq i\leq n$.

Then, we obtain the inductive limit $G_{(\infty)}=\varinjlim G_{(m)}$ of amalgamated free products  of the family $(G_{n},A_{n},\varphi_{n})_{n\in\m{Z}}$ and the positive cone $P_{(m)}$ of $G_{(m)}$, which induces an isolated ordering and contains the positive cone $P_{m}$ of $G_m$. Moreover, we obtain the positive cone $ P_{(\infty)}$ of $G_{(\infty)}$, and $P_{(\infty)}$ is the inductive limit of $\{P_{(m)}\}_{m\In}$.
\end{lem} 

\begin{proof}
A similar assertion is written in Ito's paper \cite{MR3476136}.

It is easy to see that the groups $G_{(\infty)}$ and $G_{(m)}$ are well-defined.

By applying Theorem \ref{thm:Ito} and, we define the positive cones $P_{(m)}$ of $G_{(m)}$ for all $m\In$ inductively.

We first define $P_{(0)}$ as $P_{{0}}$. Suppose that $P_{(m)}$ is defined for some $m\In$. In order to apply Theorem \ref{thm:Ito}, we set $G=G_{(m)}$, $H=G_{m+1}$, $z_G=z_0$, $z_H=z_{m+1}$, $\mathcal{G}=P_{(m)}$, $\mathcal{H}=P_{m+1}$. By virtue of
Theorem \ref{thm:Ito}, we have the positive cone $P'_{(m)}$ of an isolated ordering of $G'_{(m)}$. In order to apply Theorem \ref{thm:Ito}, we next set $G=G_{-m-1}$, $H=G'_{(m)}$, $z_G=z_{-m-1}$, $z_H=z_{0}$, $\mathcal{G}=P_{-m-1}$, $\mathcal{H}=P'_{(m)}$. By virtue of Theorem \ref{thm:Ito}, we have the positive cone $P_{(m+1)}$ of an isolated ordering of $G_{(m+1)}$.
 
 By induction, we obtain the positive cones $P_{(m)}$ of isolated orderings, which form an inductive system by inclusions. Thus,  by Proposition \ref{prop:general} (1), $ P_{(\infty)}\coloneqq\varinjlim P_{(m)}$ is well-defined and the positive cone of a left ordering of $G_{(\infty)}$.
\end{proof}

We define $ <_{(m)}$ to be the left ordering induced by $P_{(m)}$ for each non-negative integer $m$. We have thus obtained the pair $(  G_{(\infty)},P_{(\infty)})$ of inductive limits, where $ P_{(\infty)}$ and $P_{(m)}$ are positive cones of $G_{(\infty)}$ and $G_{(m)}$, respectively. In addition, all $P_{(m)}$ are positive cones of isolated orderings. We are ready to prove one of our main theorems. 

\begin{proof}[Proof of Theorem \ref{thm:main2}] By Lemma \ref{lem:Ito}, $A_{n}=\langle z_{{n}}\rangle$, $G_{(\infty)}=\varinjlim G_{(m)}$ and $ P_{(\infty)}=\varinjlim P_{(m)}$ are well-defined. We only prove that $ P_{(\infty)}$ is approximated by its conjugate orderings. We apply Lemma \ref{lem:main} and Lemma \ref{lem:main2} (ii) to prove that $P_{(\infty)}$ induces a non-isolated ordering.

\smallskip
\paragraph{Claim 1.}$ G_{(\infty)}$ is not abelian.
\smallskip

For every integer $n$, $G_n$ is not trivial because $z_n$ is an element of $G_n$. By hypothesis, the subgroups $G_{0}, G_{1}$ and $G'_{(0)}=G_{0}*G_{1}(\langle z_{{0}}\rangle \xrightarrow{\varphi_{0}} \langle z_{{1}}\rangle)$ are abelian. Since all elements of $G_{0}$ commute with all elements of $G_{1}$ in $G_{(0)}'$, one of $G_0$ and $G_1$ includes the other. Without loss of generality, we suppose that $G_{1}$ is a subgroup of $G_{0}$, which means that $\langle z_{{1}\rangle}$ is isomorphic to $G_{1}$. We proceed the above operation inductively. It follows that $\langle z_{{n}}\rangle$ is isomorphic to $G_{n}$ for all $n\in\m{Z}_{\neq0}$, which is contradiction.

\smallskip
\paragraph{Claim 2.}$P_{(m)}$ is discrete.
\smallskip

It follows from Theorem \ref{thm:Ito}, which asserts that $P_{(m)}$ has the minimal positive element.

\smallskip
\paragraph{Claim 3.}The condition (C1) of Lemma \ref{lem:main} is satisfied.
\smallskip

Fix a non-negative integer $m$ and let $C$ be a convex subgroup of $( G_{(\infty)},P_{(\infty)})$ which contains $\Gamma^{\langle z_{{m}}\rangle}(P_{m+1})$. The equation $z_{{m}}=z_{{m+1}}$ holds in $G_{(m+1)}$. In addition, for all generators $g_{i,m+1}\in G_{m+1}$, since $z_{{m+1}}^{-1}\leq_{(m+1)}g_{i,m+1}\leq_{(m+1)}z_{{m+1}}$ holds, $C$ contains all generators by convexity. Therefore, by Proposition \ref{prop:convex} (i), we have $G_{m+1}\subset \Gamma^{\langle z_{{m}}\rangle}(P_{m+1})$.

\smallskip
\paragraph{Claim 4.}The condition (C2) of Lemma \ref{lem:main} is satisfied.
\smallskip

The proof is similar to that of Claim 3.

\smallskip
\paragraph{Claim 5.}The condition (C3) of Lemma \ref{lem:main} is satisfied.
\smallskip

We prove the subclaim below.

\smallskip
\subparagraph{Subclaim.} The left ordering induced by $P_{(\infty)}$ is dense.
\smallskip

Suppose that $ P_{(\infty)}$ is discrete, that is, it has the minimal positive element $p_{\infty}$ which is contained in $P_{(m)}$. Let $p_m$ be the minimal positive element of $(G_{(m)},P_{(m)})$. Then, $p_{\infty}$ equals to $p_{m}$. By proof of Lemma \ref{lem:Ito}, $P'_{(m)}$ has the minimal positive element denoted by $p'_{m}$ which satisfies the following equation by Theorem \ref{thm:Ito}:
\[
p_{\infty}\leq_{(m)'}p'_{m}\leq_{(m)'}p_{m}=p_{\infty},
\]
where $<_{(m)'}$ is the left ordering induced by $P'_{(m)}$. Namely, $p_{m}$ equals to $p'_{m}$. Since $p'_{m}=p_{m}z_{{m+1}}^{-1}g_{1,m+1}$ holds by Theorem \ref{thm:Ito}, $p_{m}=p'_{m}=p_{m}z_{{m+1}}^{-1}g_{1,m+1}$ follows, that is, $z_{{m+1}}=g_{1,m+1}$ holds. Then we have $g_{1,m+1}=g_{2,m+1}=\cdots =z_{{m+1}}$, which means that $G_{m+1}\simeq \langle z_{{m+1}}\rangle$ holds. Similarly, $P_{(m+1)}$ has the minimal positive element $p_{m+1}$, which equals to $p'_{m}$. Since $p_{m+1}=g_{1,-m-1}z_{{-m-1}}^{-1}p'_{m}$ holds by Theorem \ref{thm:Ito}, $p'_{m}=p_{m+1}=g_{1,-m-1}z_{{-m-1}}^{-1}p'_{m}$ follows, that is, $G_{-m-1}\simeq \langle z_{{-m-1}}\rangle$ holds. This means that for all $n>m$, $G_{n}=\langle z_{{n}}\rangle\simeq\langle z_{{-n}}\rangle=G_{-n}$ holds by induction, which is a contradiction.

\medskip

We next prove Claim 5. Let $C$ be a convex subgroup $C$ of $(G_{(\infty)},P_{(\infty)})$ which contains the minimal positive element $p_m$ of $(G_{(m)},P_{(m)})$ for some $m\In$. Since $P_{(\infty)}$ is dense, if we take large enough $l$, there exists a element $g\in G'_{(l)}$ with $1<_{(l)'}g<_{(l)'}p_m$, where $<_{(l)'}$ is the left ordering induced by $P'_{(l)}$ in the proof of Lemma \ref{lem:Ito}. Hence $C$ contains $g$ by convexity. It is sufficient to prove that $C$ contains $z_{l+1}$.

In what follows, we use the notation of Theorem \ref{thm:Ito}. That is, we assume that $G_{(l)}$ is generated by $g_{1},\dots,g_{s}$,  $G_{l+1}$ is generated by $h_{1},\dots,h_{t}$,  and $G'_{(l)}= G_{(l)}*G_{l+1}(\langle z_{l}\rangle\xrightarrow{\varphi_{l}}\langle z_{l+1}\rangle)$ is generated by $x_{1},\dots,x_{s},h_{1},\dots,h_{t}$, where $x_i=g_iz_{l+1}h_1$ holds for every $i\in\{1,\dots,s\}$. We suppose that each of the sets of the generators generate the positive cones
\begin{align}
    &P_{(l)}=\langle g_{1},\dots,g_{s}\rangle^+,\\
    &P_{l+1}=\langle h_{1},\dots,h_{t}\rangle^+,\\
    &P'_{(l)}=\langle x_{1},\dots,x_{s},h_{1},\dots,h_{t}\rangle^+,
\end{align}
with the isolated left orderings satisfying
\begin{align}
&1<_{(l)}g_1<_{(l)}\cdots<_{(l)}g_s,\\
&1<_{l+1}h_1<_{l+1}\cdots <_{l+1} h_t,\\
&1<_{(l)'}x_1<_{(l)'}\cdots <_{(l)'}x_s<_{(l)'}h_1<_{(l)'}\cdots<_{(l)'}h_t.
\end{align}
In addition we assume [CF($G_{(l)}$)],[CF($G_{l+1}$)],[INV($G_{l+1}$)] as in Theorem \ref{thm:Ito}.

Then $C$ contains $x_{1}$ and some other positive element $y$ of $G'_{(l)}$. Multiplying $y$ by $x_{1}^{-1}$ from the left if necessary, we have $y=y'p$, where $y'$ is a positive element and equals to a power of a element contained in $\{x_{2},\dots,x_{s},h_{1},\dots,h_{t}\}$, and $p$ is a positive element of $G'_{(l)}$. Suppose that $y<_{(l)'}x_{2}$ holds. We have $p<_{(l)'}y'^{-1}x_{2}\leq_{(l)'}1$, which is contradiction. Therefore, $C$ contains $x_{2}$ by convexity. We have $1<_{(l)'}z_{l+1}h_{2}z_{l+1}^{-1}$ because of [INV($G_{l+1}$)]. Therefore, we obtain the following inequalities:
\begin{align}
x_{1}^{-1}x_{2}z_{l+1}^{-1}&=h_{1}^{-1}z_{l+1}g_{1}^{-1}g_{2}z_{l+1}h_{2}z_{l+1}^{-1}\\
&=(h_{1}^{-1}z_{l+1})(g_{1}^{-1}g_{2})(z_{l+1}h_{2}z_{l+1}^{-1})\\
&>_{(l)'}1,
\end{align}
where the words inside the three parentheses are positive. Hence, $z_{l+1}$ is smaller than $x_{1}^{-1}x_{2}$, which means that $C$ contains $z_{l+1}$ by convexity.

\bs
Therefore, by Lemma \ref{lem:main}, this completes the proof.

\end{proof}

Theorem \ref{thm:main2} leads to a corollary similar to Corollary \ref{cor:Dehornoy} by Proposition \ref{prop:HNN}.

\begin{cor}\label{cor:Ito}
We assume the same hypothesis as Theorem \ref{thm:main2}. Then, $G_{(\infty)}$ is embedded in the HNN-extension $H$ with presentation
\begin{align}
\langle\, g,t\mid tat^{-1}=\varphi(a)\ \text{for all}\ a\in A_{m}, g\in G_{m}, m\In\,\rangle
\end{align}
Moreover, $ P_{(\infty)}$ is extended to two positive cones of left orderings of $H$, which are approximated by its conjugate orderings, i.e. which induce non-isolated orderings.
\end{cor}
\begin{proof}
It is proved by the same way as Corollary \ref{cor:Dehornoy}.
\end{proof}

We showed two examples that all $P_{(m)}$ induce isolated orderings but $ P_{(\infty)}$ dose not. This means that it is difficult to find isolated orderings of countably generated groups. Therefore, we pose the following question. 
\begin{qes}
Are there exist genuine isolated orderings of inductive limits of amalgamated free products of countably many groups?
\end{qes}

\bibliographystyle{plain}
\bibliography{202309ref}

\end{document}